\documentclass{article}
\usepackage{euscript}
\usepackage{amsmath}
\usepackage{graphicx}
\usepackage{amsthm}
\usepackage{amssymb}
\usepackage{epsfig}
\usepackage[all]{xy}
\usepackage{url}
\usepackage{tkz-berge}
\theoremstyle{theorem}
\newtheorem{theorem}{Theorem}[section]
\newtheorem{corollary}[theorem]{Corollary}

\newtheorem{lemma}[theorem]{Lemma}
\newtheorem{proposition}[theorem]{Proposition}

\newtheorem{definition}{Definition}[section]

\newtheorem{remark}{Remark}[section]

\numberwithin{equation}{section}
\title{On linear equations arising in Combinatorics (Part I)}
\author{Masood Aryapoor\footnote{E-mail: aryapoor2002@yahoo.com}}
\date{}

\begin{document}

\maketitle

\begin{abstract}

The main point of this paper is to present a class of equations over integers that one can check if they have a solution 
by checking a set of inequalities. The prototype of such equations is the equations appearing in the well-known Gale-Ryser theorem.

\end{abstract}

%%%%%%%%%%%%%%%%%%%%%%%%%%%%%%%%%%%%%%%%%%%%%%%%%%%%%%%%%%%%%%%%%%%%%%%%%%%%%%%%%%%%%%%%%%%%%%%%%%%%%%%%%%%
%%%%%%%%%%%%%%%%%%%%%%%%%%%%%%%%%%%%%%%%%%%%%%%%%%%%%%%%%%%%%%%%%%%%%%%%%%%%%%%%%%%%%%%%%%%%%%%%%%%%%%%%%%%
%%%%%%%%%%%%%%%%%%%%%%%%%%%%%%%%%%%%%%%%%%%%%%%%%%%%%%%%%%%%%%%%%%%%%%%%%%%%%%%%%%%%%%%%%%%%%%%%%%%%%%%%%%%
%%%%%%%%%%%%%%%%%%%%%%%%%%%%%%%%%%%%%%%%%%%%%%%%%%%%%%%%%%%%%%%%%%%%%%%%%%%%%%%%%%%%%%%%%%%%%%%%%%%%%%%%%%%

\begin{section}{Introduction}

This work is partly motivated by \cite{MA}. In his work, E. S. Mahmoodian successfully uses a single method, 
called the critical case method, to prove a number of well-known existential results in     
combinatorics, such as Berge's theorem on the existence of a matching of deficiency $d$, Tutte's theorem 
on the existence of 1-factors, the Gale-Ryser theorem, the Erdos-Gallai theorem, and Landau's theorem. 
The central question in this paper is the following seemingly philosophical question.
What is so special for such results that the existence of a  class of combinatorial object can be decided by a set of inequalities?    
Let me formulate the problem concretely using a somewhat general setting. Suppose that an 
$n\times m$ integral matrix $A$ and vectors $B\in   \mathbb{Z}^n$, $C,D\in \mathbb{Z}^m$ are given.   
Consider the following system of equations and inequalities 
\begin{equation}\label{system}
A X=B, C\leq X\leq D
\end{equation}
where $\begin{pmatrix}
a_1\\
\vdots\\
a_{m}
\end{pmatrix}\leq \begin{pmatrix}
b_1\\
\vdots\\
b_{m}
\end{pmatrix}$ means $a_i\leq b_i$ for every $i=1,...,m$. 
It is folkloric that  the problem of the existence of certain combinatorial objects, including the ones mentioned above, can be formulated as
the existence of an integral solution for a system of the above form. The heuristic question is then: Under what circumstances can the existence of an integral solution for 
system \ref{system} be answered by some  "reasonable" arithmetical conditions and inequalities? In fact, it is easy to derive such conditions which in general are only necessary.  
More precisely, one can easily see that system \ref{system} has a solution over integers only if 
(1) the equation $A X=B$ has a solution over integers, and (2) the system  has a solution over real numbers. Now, part (1) gives us a set of arithmetical conditions 
(here the relevant concept is the concept of smith normal forms). Part (2), belonging to the subject of Linear programming, gives rise to a collection of conditions in form of inequalities. 

Having been armed with the conditions, obtained by (1) and (2), one may wonder if these
 conditions are also sufficient for the existence of an integral solution for his/her favorite system. 
The main goal of this paper is to characterize those matrices $A$ for which this is the case.
Using this characterizations, one can in fact prove that a class of systems coming from combinatorics, such as the Gale-Ryser theorem and Landau's theorem among others,
are of this form, and perhaps not surprisingly, the corresponding necessary and sufficient conditions, turns out to be the familiar ones.  
However, the real importance of this characterization is that it 
gives a tool to check if the existence of a given combinatorial object can be decided by a set of arithmetical conditions and a set of inequalities of a similar nature.

It is, in fact, a very restrictive property for an integral matrix $A$ that the conditions given by (1) and (2), are sufficient for the existence of an integral solution for system \ref{system}. 
To obtain more powerful results, one therefore needs to introduce extra conditions. This issue will be pursued in the next paper.

\end{section}

%%%%%%%%%%%%%%%%%%%%%%%%%%%%%%%%%%%%%%%%%%%%%%%%%%%%%%%%%%%%%%%%%%%%%%%%%%%%%%%%%%%%%%%%%%%%%%%%%%%%%%%%%%%
%%%%%%%%%%%%%%%%%%%%%%%%%%%%%%%%%%%%%%%%%%%%%%%%%%%%%%%%%%%%%%%%%%%%%%%%%%%%%%%%%%%%%%%%%%%%%%%%%%%%%%%%%%%
%%%%%%%%%%%%%%%%%%%%%%%%%%%%%%%%%%%%%%%%%%%%%%%%%%%%%%%%%%%%%%%%%%%%%%%%%%%%%%%%%%%%%%%%%%%%%%%%%%%%%%%%%%%
%%%%%%%%%%%%%%%%%%%%%%%%%%%%%%%%%%%%%%%%%%%%%%%%%%%%%%%%%%%%%%%%%%%%%%%%%%%%%%%%%%%%%%%%%%%%%%%%%%%%%%%%%%%

\begin{section}{Farkas' Lemma}

In this section, the relevant material from Linear Programming is discussed.  
The notation $(u,v)$ stands for the standard inner product of two vectors $u,v\in \mathbb{R}^n$.

%%%%%%%%%%%%%%%%%%%%%%%%%%%%%%%%%%%%%%%%%%%%%%%%%%%%%%%%%%%%%%%%%%%%%%%%%%%%%%%%%%%%%%%%%%%%%%%%%%%%%%%%%%%

\begin{subsection}{Farkas' Lemma over $\mathbb{R}$}

We start with the following version of Farkas' lemma whose proof is given for the sake of completeness.  
 
\begin{lemma}  \label{Farkas, real}

Let $v_1,...,v_m\in \mathbb{R}^n$ and let $a_1\leq b_1,...,a_m\leq b_m$ be arbitrary real numbers. Then a vector $w\in   \mathbb{R}^n$ can be written as 
$w=\sum_{i=1}^m x_i v_i$ for some real numbers $a_1\leq x_1\leq  b_1,...,a_m\leq x_m\leq b_m$ if and only if for every vector $u\in \mathbb{R}^n$, we have

\begin{equation}\label{inequality}
(u,w)\leq \sum_{i=1}^m a_i\frac{(u,v_i)-|(u,v_i)|}{2}+\sum_{i=1}^m b_i\frac{(u,v_i)+|(u,v_i)|}{2}.
\end{equation}

\end{lemma}

\begin{proof}

First we prove the "only if" direction. So suppose that we have $w=\sum_{i=1}^m x_i v_i$ for some real numbers 
$a_1\leq x_1\leq  b_1,...,a_m\leq x_m\leq b_m$. Then for every $u\in \mathbb{R}^n$, we have 
$$(u,w)=\sum_{i=1}^m x_i(u,v_i)=\sum_{i=1}^m x_i\frac{(u,v_i)-|(u,v_i)|}{2}+\sum_{i=1}^m x_i\frac{(u,v_i)+|(u,v_i)|}{2}$$
$$\leq \sum_{i=1}^m a_i\frac{(u,v_i)-|(u,v_i)|}{2}+\sum_{i=1}^m b_i\frac{(u,v_i)+|(u,v_i)|}{2}.$$

Conversely, assume that the condition holds for a given vector $w\in \mathbb{R}^n$. On the contrary suppose that $w$ is not equal to
$\sum_{i=1}^m x_i v_i$ for every choice of real numbers $a_1\leq x_1\leq  b_1,...,a_m\leq x_m\leq b_m$. Let $C$ be the set of vectors 
$\sum_{i=1}^m y_i v_i$ where $a_1\leq y_1\leq  b_1,...,a_m\leq y_m\leq b_m$ are arbitrary real numbers. Since $C$ is the image of the compact set 
$[a_1,b_1]\times\dots\times [a_m,b_m]$ under the continuous map $(y_1,...,y_m)\mapsto \sum_{i=1}^m y_i v_i$, the set $C$ is a compact set. 
Furthermore, it is easy to see that $C$ is convex. Now, since $C$ is a compact convex set and $w\notin C$,
 by the hyperplane separation Lemma, there exists a vector $u_0\in \mathbb{R}^n$ such that for every $w'\in C$, we have 
$(u_0,w')<(u_0,w)$. This means that for all real numbers $a_1\leq y_1\leq  b_1,...,a_m\leq y_m\leq b_m$, we have
$$\sum_{i=1}^m y_i(u_0,v_i)=(u_0,\sum_{i=1}^m y_i v_i )<(u_0,w). $$ 
But,  setting $y_i=b_i$ if $(u_0,v_i)\geq 0$ and $y_i=a_i$ if $ (u_0,v_i)<0$, in this   inequality, gives us 
$$\sum_{i=1}^m a_i\frac{(u_0,v_i)-|(u_0,v_i)|}{2}+\sum_{i=1}^m b_i\frac{(u_0,v_i)+|(u_0,v_i)|}{2}<(u_0,w),$$
which is a contradiction.

\end{proof}

In order to apply Lemma  \ref{Farkas, real}, we need to check if Inequality \ref{inequality} holds for \underline{all vectors} $u\in \mathbb{R}^n$. 
However, it turns out that if this inequality holds for certain vectors in $\mathbb{R}^n$, then it holds for all vectors.
To examine this issue closely, suppose that   $w,v_1,...,v_m\in \mathbb{R}^n$  are given as in Lemma \label{Farkas, real}. 
Set $V=\sum_{i=1}^m\mathbb{R} v_i$ and let $V^{\bot}$ denote the set of all vectors $u\in\mathbb{R}^n$ such that 
$(u,v_i)=0$ for all $i=1,...,m$.  We note that if $u_0\in V^{\bot}$, then Inequality \ref{inequality} holds for $u=u_0$ and $u=-u_0$ if and only if
$(u_0,w)=0$. So if we choose a basis $w_1,...,w_k\in \mathbb{R}^n$ for the vector space $V^{\bot}$ over $\mathbb{R}$,
then \ref{inequality} holds for all $u\in V^{\bot}$ if and only if $(w_i,w)=0$ for every $i=1,...,k$. 

An arbitrary vector $u_0\in \mathbb{R}^n$
can be written as $u_0=u_0'+u_0''$ where $u_0'\in V$ and $u_0''\in V^{\bot}$. It is easy to see that if Inequality \ref{inequality} holds for 
$u=u_0'$ and $u=u_0''$, then it holds for $u_0$ as well. 
The above discussion leads us to our ''first reduction'':

\noindent\textbf{First reduction:} Inequality \ref{inequality} holds for all  vectors  $u\in \mathbb{R}^n$, 
if and only if it holds for $u=w_1,-w_1,...,w_k,-w_k$, and all vectors $u\in\sum_{i=1}^m\mathbb{R} v_i$. 
Moreover the inequality holds for $u=w_1,-w_1,...,w_k,-w_k$, if and only if $(w_i,w)=0$ for every $i=1,...,k$, if and only if $w\in  \sum_{i=1}^m\mathbb{R} v_i$. 

Now suppose that the inequality holds for some $u\in \mathbb{R}^n$  and let $u=r u'$ where $u' \in \mathbb{R}^n$ and $r$ is a positive real number. 
Then we have 
$$r (u',w)=(u,w)\leq \sum_{i=1}^m a_i\frac{(u,v_i)-|(u,v_i)|}{2}+\sum_{i=1}^m b_i\frac{(u,v_i)+|( u,v_i)|}{2}=$$
$$ r(\sum_{i=1}^m a_i\frac{( u',v_i)-|(u',v_i)|}{2}+\sum_{i=1}^m b_i\frac{(u',v_i)+|(u',v_i)|}{2}).$$
Since $r$ is positive, it follows that the inequality holds for $w'$ as well. So our ''second reduction'' is the following:

\noindent \textbf{Second reduction:}  Inequality \ref{inequality} holds for a  vector  $u\in \mathbb{R}^n$, if and only if
it holds for $r u$ where $r$ is an arbitrary positive real number. 

In order to elaborate the second reduction, we introduce a few definitions. Consider the following equivalence relation on 
$\mathbb{R}^n\setminus\{0\}$: two vectors $v,v'\in\mathbb{R}^n\setminus\{0\}$ are equivalent if $v'=r v$  for a positive real number $r$. It is easy to see that this is in fact 
an equivalence relation. Define  $\mathbb{R}\mathbb{P}_+^{n-1}$ to be set of all the equivalent classes of this equivalence relation.
The equivalence class containing $v\in\mathbb{R}^n\setminus\{0\}$ is denoted by $[v]$ and elements of $\mathbb{R}\mathbb{P}_+^{n-1}$ are called points.     
 
\begin{definition}

Assume that vectors $v_1,...,v_m\in \mathbb{R}^n$ are given. 
A nonzero vector $u\in \sum_{i=1}^m\mathbb{R} v_i$ is called $\{v_1,...,v_m\}$--decomposable (or decomposable with respect to  
$v_1,...,v_m$)
if there exist nonzero vectors $u',u''\in\sum_{i=1}^m\mathbb{R} v_i$ such that $u=u'+u''$, $[u'],[u'']\neq [u]$  and $(u',v)(u'',v)\geq 0$ for every $v\in\{v_1,...,v_m\}$. 
If a nonzero vector in $\sum_{i=1}^m\mathbb{R} v_i$ is not $\{v_1,...,v_m\}$--decomposable, 
 it is called $\{v_1,...,v_m\}$--indecomposable (or indecomposable with respect to  
$v_1,...,v_m$).

\end{definition}

It is clear from the definition that a vector $u\in\sum_{i=1}^m\mathbb{R} v_i$ is $\{v_1,...,v_m\}$--decomposable
 if and only if $ r u$ is $\{v_1,...,v_m\}$--decomposable for some positive real 
number $r$. In other words, if $[u]=[u_1]\in\mathbb{R}\mathbb{P}_+^{n-1} $, then $u$ is  $\{v_1,...,v_m\}$--decomposable 
if and only if $u_1$ is $\{v_1,...,v_m\}$--decomposable. 
A point $x\in\mathbb{R}\mathbb{P}_+^{n-1} $ is  called $\{v_1,...,v_m\}$--decomposable if $x=[u]$ 
 for some $\{v_1,...,v_m\}$--decomposable vector $u\in \sum_{i=1}^m\mathbb{R} v_i$.
By the above argument, this definition is  well-defined, i.e. it does not depend on $u$. 
In a similar way, we define $\{v_1,...,v_m\}$--indecomposable points in $\mathbb{R}\mathbb{P}_+^{n-1} $.

We want to characterize  the  $\{v_1,...,v_m\}$--indecomposable points in $\mathbb{R}\mathbb{P}_+^{n-1} $. To do so, we introduce some notations. 
Given a set $I\subset \{1,...,m\}$, we define the vector space $V(I)\subset \mathbb{R}^n $
to be the set of all vectors $u\in\sum_{i=1}^m\mathbb{R} v_i$ such that $(u,v_i)=0$ for every $i\in I$. It is clear that $V(\{1,...,m\})=\{0\}$ and $V(I)\subset V(J)$ if $J\subset I$. 
Moreover, for a nonzero vector $u\in\sum_{i=1}^m\mathbb{R} v_i$, 
there is a unique set $I_u\subset \{1,...,m\}$, such that $u\in V(I_u)$, but $u\notin V(J)$ for every $I_u\subsetneq J\subset \{1,...,m\}$.
In fact, we have $I_u=\{i\in\{1,...,m\}|(u,v_i)=0\}$. 
Using these notations, we can state the following lemma.  

\begin{lemma}\label{indecomposable, characterization}

(1) A vector $u\in\sum_{i=1}^m\mathbb{R} v_i$ is $\{v_1,...,v_m\}$--indecomposable if and only if $V(I_u)$, as a vector space over  $\mathbb{R}$, has dimension one.\\
(2) The number  of $\{v_1,...,v_m\}$--indecomposable points in  $\mathbb{R}\mathbb{P}_+^{n-1} $ is finite. Moreover 
every nonzero vector   $u\in \sum_{i=1}^m\mathbb{R} v_i$ can be written as $u=u_1+\cdots+u_l$, where $u_i$'s are 
$\{v_1,...,v_m\}$--indecomposable, such that  $(u,v)(u_i,v)\geq 0$ for every $i=1,...,l$ and every $v\in\{v_1,...,v_m\}$.

\end{lemma}

\begin{proof}

(1) To prove the "if direction", suppose that $u\in\sum_{i=1}^m\mathbb{R} v_i$ is $\{v_1,...,v_m\}$--decomposable. 
We need to show that $V(I_u)$  has dimension greater than one.
Since $u\in\sum_{i=1}^m\mathbb{R} v_i$ is $\{v_1,...,v_m\}$--decomposable,
there exist nonzero vectors $u',u''\in\sum_{i=1}^m\mathbb{R} v_i$ such that $u=u'+u''$, $[u'],[u'']\neq [u]$  and $(u',v)(u'',v)\geq 0$ for  every $v\in\{v_1,...,v_m\}$. 
Since  $(u,v)=(u',v)+(u'',v)$ and $ (u',v)(u'',v)\geq 0$ for  every $v\in\{v_1,...,v_m\}$, 
we must have $u',u''\in V(I_u)$. To show that $V(I_u)$  has dimension greater than one, it is enough 
to prove that $u,u'$ (or $u,u''$) are linearly independent over $\mathbb{R}$. If, on the contrary, $u,u'$ (and $u,u''$) are linearly dependent over $\mathbb{R}$, 
then $u'=r' u$ and $u''=r'' u$ for some nonzero real numbers $r',r''$. But then the conditions $u=u'+u''$   and $(u',v)(u'',v)\geq 0$ for  every $v\in\{v_1,...,v_m\}$, 
imply that $r'+r''=1$ and $ r' r''(u,v)^2\geq 0$ for  every $v\in\{v_1,...,v_m\}$. If $(u,v)=0$ for all  $v\in\{v_1,...,v_m\}$, then,
since $u\in  \sum_{i=1}^m\mathbb{R} v_i$, we would have $(u,u)=0$, i.e. $u=0$, a contradiction. So there is some 
$v_0\in\{v_1,...,v_m\}$ for which $(u,v_0)\neq 0$ and consequently, by $ r' r''(u,v_0)^2\geq 0$, we conclude that  $r' r''>0$. So, the  
numbers $r'$ and $r''$ are positive, which in turn implies that $[u']=[u'']=[u]$, a contradiction. This proves the "if direction''. 

To prove the other direction, suppose that $u$ is $\{v_1,...,v_m\}$--indecomposable, but   
$V(I_u)$ has dimension greater than one. 
In the vector space $V(I_u)$, we consider the "chamber" $C$ consisting of all vectors $\alpha\in V(I_u)$ such that 
for every $v\in \{v_1,...,v_m\}$, we have $(\alpha,v)> 0$, if $(u,v)>0$, and    
$ (\alpha,v)< 0$ if $(u,v)<0 $. The set $C$ is nonempty since $u\in C$. It is easy to see that $C$ is 
an open subset of the vector space $V(I_u)$ (where the topology is just the induced topology from $\mathbb{R}^n$). 
I claim that there exist a nonzero vector in $\bar{C}\setminus C$, where $\bar{C}$ is the closure of $C$ in $V(I_u)$.
To show this, we choose a nonzero vector  $\alpha \in V(I_u)$ such that $u$ and $\alpha$ are $\mathbb{R}$--linearly independent. 
This is possible because the dimension of $V(I_u)$ is greater than one. We may assume that $\alpha\notin C$, since otherwise we take $-\alpha$ which satisfies $-\alpha\notin C$. 
Now, it is easy to see that 
the set $\{t\in [0,1]| (1-t) u+t \alpha\in C\}$ is an open connected subset of the interval $[0,1]$, containing $0$, but not $1$. So it must be of form $[0,t_0)$ for some $0<t_0< 1$. 
It then follows that $\alpha_0=(1-t_0) u+t_0 \alpha$ is a nonzero vector on the boundary of $C$, i.e. $\alpha_0\in \bar{C}\setminus C$.

Note  that we have 
$(\alpha_0,v)\geq 0$, if $(u,v)>0$, and    $ (\alpha_0,v)\leq 0$ if $(u,v)<0$ for every $v\in \{v_1,...,v_m\}$. It follows 
that there is a positive real number $r$, small enough such that $(\alpha_0,v)(u,v) \geq r(\alpha_0,v)^2$ for  every $v\in\{v_1,...,v_m\}$, which is equivalent to
$(r \alpha_0,v)(u-r \alpha_0,v)\geq 0$ for  every $v\in\{v_1,...,v_m\}$. Since $\alpha_0\notin C$ and $\alpha_0\neq 0$, 
there is some $v_j\in \{v_1,...,v_m\}$ such that $(u,v_j)\neq 0$, but $(\alpha_0,v_j)=0$. In particular 
we have $\alpha_0\notin \mathbb {R} u$ from which it follows that $[u]\neq [r \alpha_0]$ and $[u]\neq [u-r \alpha_0]$.  Since $u=r \alpha_0+(u-r \alpha_0)$, 
we conclude, from the above facts, that the vector $u$ is $\{v_1,...,v_m\}$--decomposable, a contradiction. This finishes the proof.  

(2)  First note that if $V(I_u)=V(I_{u'})$ for two $\{v_1,...,v_m\}$--indecomposable vectors $u$ and $u'$,
 then $u=r u'$ for some nonzero real number $r$, because, we have 
$u\in V(I_u), u'\in V(I_{u'})$, and by part (1), the vector spaces $V(I_u)$ and $V(I_{u'})$ are one dimensional. 
On the other hand, a point $x=[u]\in \mathbb{R}\mathbb{P}_+^{n-1}$ is  $\{v_1,...,v_m\}$--indecomposable if and only if $V(I_u)$ is one-dimensional.
From these observations, we conclude that    the number  of $\{v_1,...,v_m\}$--indecomposable points in  $\mathbb{R}\mathbb{P}_+^{n-1} $ is at most twice the number of 
sets $I\subset \{1,...,m\}$ for which $V(I)$ is one-dimensional. In particular, this number is finite.

To prove the second statement, we use induction on the dimension $d$ of the vector space $V(I_u)$. 
 If $d=0$, then $u=0$ and there is nothing to prove. 
If $d=1$, then  by part (1), $u$ is $\{v_1,...,v_m\}$--indecomposable and we are done. 
Now suppose that $d>1$.  We consider the chamber $C$ and the vector $\alpha_0$, used in part (1). 
Note that we have $V(I_{\alpha_0})\subset V(I_u)$, because $\alpha_0\in V(I_u)$. As we have seen, there is some
 $v_j\in \{v_1,...,v_m\}$ such that $(u,v_j)\neq 0$, but $(\alpha_0,v_j)=0$, i.e. $u\notin V(I_{\alpha_0})$. In particular, it follows that  $V(I_{\alpha_0})$
has dimension less than $d$. 

Now, let $r_0$ be the minimum of the 
numbers $\frac{(u,v)}{(\alpha_0,v)}$ where $v\in \{v_1,...,v_m\}$ with $(\alpha_0,v)\neq 0$. Since both $u$ and $\alpha_0$ belong
 to   $\bar{C}$, we see that $r_0$ is positive.  Furthermore, by the choice of $r_0$, 
we have, $ (r_0 \alpha_0,v)(u-r_0 \alpha_0,v)\geq 0$ for  every $v\in\{v_1,...,v_m\}$.
Since $r_0=  \frac{(u,v_k)}{(\alpha_0,v_k)}$ for some $v_k\in \{v_1,...,v_m\}$  (which implies $(u,v_k)\neq 0$, 
but  $(u-r_0 \alpha_0,v_k)=0$), we see that $V(I_{u-r_0\alpha_0})$ has dimension less than $d$. Note that $u-r_0\alpha_0\neq 0$, since, as we have 
seen $\alpha_0\notin \mathbb {R} u$. 

Having proved that both vectors spaces $V(I_{\alpha_0})$ and $V(I_{u-r_0\alpha_0})$ have dimensions less than $d$, 
we can use induction, to obtain desirable presentations $r_0\alpha_0=u'_1+\cdots+u'_{l'}$ and $u-r_0\alpha_0=u''_1+\cdots+u''_{l''}$.
I claim that $u=u'_1+\cdots+u'_{l'}+u''_1+\cdots+u''_{l''}$ is the desired presentation for $u$. For every $i=1,...,l'$, and every $v\in\{v_1,...,v_m\}$, we have,
$$((u,v)(u'_i,v))((u'_i,v)(r_0\alpha_0,v))=(u,v)(r_0\alpha_0,v)(u'_i,v)^2$$
$$=(r_0\alpha_0,v)(u-r_0\alpha_0,v)(u'_i,v)^2+(r_0\alpha_0,v)^2(u'_i,v)^2\geq 0.$$
Now we use the fact that the presentation $r_0\alpha_0=u'_1+\cdots+u'_{l'}$ has the corresponding 
properties. So, either $(u'_i,v)(r_0\alpha_0,v)>0$, in which case we must have  $(u,v)(u'_i,v)\geq 0$, by the above inequality, 
or  $(u'_i,v)(r_0\alpha_0,v)=0$, in which case we must have $(u'_i,v)=0$, 
and consequently $(u,v)(u'_i,v)\geq 0$. Similarly, one can show that 
$(u,v)(u''_i,v)\geq 0$ for every $i=1,...,l''$, and every $v\in\{v_1,...,v_m\}$. Hence the proof is  complete. 
 
\end{proof}

Using the above discussion, we can restate Lemma \ref{Farkas, real} in the following way.  

\begin{theorem}  \label{Farkas, real and indecomposable}

Let $v_1,...,v_m\in \mathbb{R}^n$ be arbitrary vectors.  
Furthermore assume that $a_1\leq b_1,...,a_m\leq b_m$ are given real numbers.   
Then a vector $w\in   \mathbb{R}^n$ can be written as 
$w=\sum_{i=1}^m x_i v_i$ for some real numbers $a_1\leq x_1\leq  b_1,...,a_m\leq x_m\leq b_m$ if and only if  $w\in  \sum_{i=1}^m\mathbb{R} v_i,$ and
$$(u,w)\leq \sum_{i=1}^m a_i\frac{(u,v_i)-|(u,v_i)|}{2}+\sum_{i=1}^m b_i\frac{(u,v_i)+|(u,v_i)|}{2},$$
for   every  $\{v_1,...,v_m\}$--indecomposable point $[u]\in \mathbb{R}\mathbb{P}_+^{n-1} $. 

\end{theorem}

\begin{proof}

Using Lemma \ref{Farkas, real} and "First reduction", we only need to prove that
if the inequality holds for every $u=u_1,...,u_l$, then 
it holds for all  vectors in $\mathbb{R}v_1+\cdots+\mathbb{R}v_m$. 
So let $u\in\mathbb{R}v_1+\cdots+\mathbb{R}v_m$ be nonzero. By Lemma  \ref{indecomposable, characterization}, 
$u$ can be written as $u=u_1+\cdots+u_l$, where the points $[u_1],...,[u_l]$ are 
$\{v_1,...,v_m\}$--indecomposable, such that  $(u,v)(u_i,v)\geq 0$ for every $i=1,...,s$ and every $v\in\{v_1,...,v_m\}$.  The former property of $u_1,...,u_l$ implies that 
$(u,v)=\sum_{i=1}^{l}(u_i,v)$ for every vector $v\in \mathbb{R}^n$ and 
the latter property of $u_1,...,u_l$ implies that  $|(u,v)|=\sum_{i=1}^{l}|(u_i,v)|$ for $v\in\{v_1,...,v_m\}$. Since the inequality holds for 
$\{v_1,...,v_m\}$--indecomposable points $[u_1],...,[u_l]$,  it holds, by "Second Reduction", for $u_1,...,u_l$, and therefore, we have
$$(u,w)=\sum_{j=1}^{l}(u_j,w)\leq \sum_{j=1}^{l}( \sum_{i=1}^m a_i\frac{(u_j,v_i)-|(u_j,v_i)|}{2}+\sum_{i=1}^m b_i\frac{(u_j,v_i)+|(u_j,v_i)|}{2})$$
$$\leq   \sum_{i=1}^m a_i\frac{\sum_{j=1}^{l}(u_j,v_i)-\sum_{j=1}^{l}|(u_j,v_i)|}{2}+\sum_{i=1}^m b_i\frac{\sum_{j=1}^{l}(u_j,v_i)+\sum_{j=1}^{l}|(u_j,v_i)|}{2}$$
$$=\sum_{i=1}^m a_i\frac{(u,v_i)-|(u,v_i)|}{2}+\sum_{i=1}^m b_i\frac{(u,v_i)+|(u,v_i)|}{2},$$
i.e. the inequality holds for $u$ as well, and we are done.

\end{proof}

\end{subsection}

%%%%%%%%%%%%%%%%%%%%%%%%%%%%%%%%%%%%%%%%%%%%%%%%%%%%%%%%%%%%%%%%%%%%%%%%%%%%%%%%%%%%%%%%%%%%%%%%%%%%%%%%%%%
%%%%%%%%%%%%%%%%%%%%%%%%%%%%%%%%%%%%%%%%%%%%%%%%%%%%%%%%%%%%%%%%%%%%%%%%%%%%%%%%%%%%%%%%%%%%%%%%%%%%%%%%%%%

\begin{subsection}{Farkas' Lemma over $\mathbb{Q}$}

In this part, we prove that Theorem  \ref{Farkas, real and indecomposable} holds over rational numbers. More precisely, we have the following  result.

\begin{theorem} \label{Farkas, rational}

Let $v_1,...,v_m\in \mathbb{Q}^n$ be some vectors and suppose that $a_1\leq b_1,...,a_m\leq b_m$ are arbitrary rational numbers.   
Then a vector $w\in   \mathbb{Q}^n$ can be written as 
$w=\sum_{i=1}^m y_i v_i$ for some rational numbers $a_1\leq y_1\leq  b_1,...,a_m\leq y_m\leq b_m$ if and only if  
$w\in  \sum_{i=1}^m\mathbb{Q} v_i,$ and
$$(u,w)\leq \sum_{i=1}^m a_i\frac{(u,v_i)-|(u,v_i)|}{2}+\sum_{i=1}^m b_i\frac{(u,v_i)+|(u,v_i)|}{2},$$
 for   every  $\{v_1,...,v_m\}$--indecomposable point $[u]\in \mathbb{R}\mathbb{P}_+^{n-1} $.

\end{theorem}

\begin{proof}

The "only if" direction follows directly from  Theorem  \ref{Farkas, real and indecomposable}. To prove the other direction, assume that the condition holds for 
a given vector $w\in \mathbb{Q}^n$. By Theorem \ref{Farkas, real and indecomposable}, 
there exist real numbers  $a_1\leq x_1\leq  b_1,...,a_m\leq x_m\leq b_m$ such that 
$w=\sum_{i=1}^m x_i v_i$. If $x_1,...,x_m$ are rational numbers, then we are done. So suppose that 
some of the numbers $x_1,...,x_m$ are not rational. Without loss of generality, we assume that $x_1,...,x_r$ are not rational, but $x_{r+1},...,x_m$ are rational. 
Set $w'=w-\sum_{i=r+1}^{m}x_i v_i$, and
let $P$ denote the set of all vectors $(z_1,...,z_r)\in\mathbb{R}^r $ such that $\sum_{i=1}^{r}z_iv_i=0$. 
The set $P$ is a vector space over $\mathbb{R}$ and since   $v_1,...,v_r\in \mathbb{Q}^n$,
the vector space $P$ has a basis $\alpha_1,...,\alpha_s\in \mathbb{Q}^r$ over $\mathbb{R}$. 
On the other hand, since $w'\in \mathbb{Q}^n$, and  $w'=\sum_{i=1}^{r}x_iv_i$, we must have 
$w'=\sum_{i=1}^{r}q_i v_i$ for some rational numbers $q_1,...,q_r$. Now, we have $(x_1,...,x_r)-(q_1,...,q_r)\in P$ and therefore 
$$(x_1,...,x_r)=(q_1,...,q_r)+t_1\alpha_1+\cdots+t_s \alpha_s,$$
for some numbers $t_1,...,t_s\in \mathbb{R}$. Since $a_i<x_i<b_i$, for every $i=1,...,r$, 
we can choose rational numbers $p_i$, close enough to $t_i$, such that
the rational numbers $y_1,...,y_r$, defined via, 
$$(y_1,...,y_r)=(q_1,...,q_r)+p_1\alpha_1+\cdots+p_s \alpha_s,$$
satisfy  $a_i<y_i<b_i$, for every $i=1,...,r$. It is then easy to see that we have $w=\sum_{i=1}^r y_iv_i+\sum_{i=r+1}^m x_i v_i$ 
which is the desired presentation.

\end{proof}

\begin{remark}

One can show that, if $v_1,...,v_m\in  \mathbb{Q}^n$, as in the above theorem,  then 
 the vectors $u_1,...,u_l$ (such that $[u_1],...,[u_l]$ give us all the $\{v_1,...,v_m\}$--indecomposable points in  $\mathbb{R}\mathbb{P}_+^{n-1} $)
can be chosen to be in $\mathbb{Q}^n$. 

\end{remark}

\end{subsection}

\end{section}

%%%%%%%%%%%%%%%%%%%%%%%%%%%%%%%%%%%%%%%%%%%%%%%%%%%%%%%%%%%%%%%%%%%%%%%%%%%%%%%%%%%%%%%%%%%%%%%%%%%%%%%%%%%
%%%%%%%%%%%%%%%%%%%%%%%%%%%%%%%%%%%%%%%%%%%%%%%%%%%%%%%%%%%%%%%%%%%%%%%%%%%%%%%%%%%%%%%%%%%%%%%%%%%%%%%%%%%
%%%%%%%%%%%%%%%%%%%%%%%%%%%%%%%%%%%%%%%%%%%%%%%%%%%%%%%%%%%%%%%%%%%%%%%%%%%%%%%%%%%%%%%%%%%%%%%%%%%%%%%%%%%
%%%%%%%%%%%%%%%%%%%%%%%%%%%%%%%%%%%%%%%%%%%%%%%%%%%%%%%%%%%%%%%%%%%%%%%%%%%%%%%%%%%%%%%%%%%%%%%%%%%%%%%%%%%
%%%%%%%%%%%%%%%%%%%%%%%%%%%%%%%%%%%%%%%%%%%%%%%%%%%%%%%%%%%%%%%%%%%%%%%%%%%%%%%%%%%%%%%%%%%%%%%%%%%%%%%%%%%
%%%%%%%%%%%%%%%%%%%%%%%%%%%%%%%%%%%%%%%%%%%%%%%%%%%%%%%%%%%%%%%%%%%%%%%%%%%%%%%%%%%%%%%%%%%%%%%%%%%%%%%%%%%
%%%%%%%%%%%%%%%%%%%%%%%%%%%%%%%%%%%%%%%%%%%%%%%%%%%%%%%%%%%%%%%%%%%%%%%%%%%%%%%%%%%%%%%%%%%%%%%%%%%%%%%%%%%
%%%%%%%%%%%%%%%%%%%%%%%%%%%%%%%%%%%%%%%%%%%%%%%%%%%%%%%%%%%%%%%%%%%%%%%%%%%%%%%%%%%%%%%%%%%%%%%%%%%%%%%%%%%

\begin{section}{Farkas' lemma over $\mathbb{Z}$}

We want to obtain a version of Farkas' lemma over $\mathbb{Z}$ similar to Theorem  \ref{Farkas, rational}. 
It is obvious that this theorem does not longer hold over $\mathbb{Z}$ in its full generality. 
So in order to have a version over integers, we need to impose some extra conditions. 
In fact, by examining  Theorem  \ref{Farkas, rational}, one is led to introduce the following condition/definition. 

\begin{definition} \label{Farkas--related}

Vectors  $v_1,...,v_m\in \mathbb{Z}^n$ are said to be Farkas--related if the following condition holds: 
For arbitrary integers $a_1\leq b_1,...,a_m\leq b_m$,
if a vector $w\in  \sum_{i=1}^m \mathbb{Z} v_i$ can be written as 
$w=\sum_{i=1}^m x_i v_i$ for some rational numbers $a_1\leq  x_1\leq b_1,...,a_m\leq x_m\leq b_m$, then 
we have $w=\sum_{i=1}^m y_i v_i$ for some integers $a_1\leq  y_1\leq b_1,...,a_m\leq y_m\leq b_m$.
 
\end{definition}

Using this definition and  Theorem \ref{Farkas, rational},  one can easily derive the following version of Farkas' Lemma over $\mathbb{Z}$.

\begin{theorem}  \label{Farkas, integers, first}

Suppose that  vectors $v_1,...,v_m\in \mathbb{Z}^n$  are  Farkas--related and let  arbitrary integers $a_1\leq b_1,...,a_m\leq b_m$ 
be given.  Then a vector $w\in  \mathbb{Z}^n$ can be written as 
$w=\sum_{i=1}^m x_i v_i$ for some integers $a_1\leq x_1\leq  b_1,...,a_m\leq x_m\leq b_m$ if and only if 
$w\in  \sum_{i=1}^m\mathbb{Z} v_i$ and
$$(u,w)\leq \sum_{i=1}^m a_i\frac{(u,v_i)-|(u,v_i)|}{2}+\sum_{i=1}^m b_i\frac{(u,v_i)+|(u,v_i)|}{2},$$
for   every  $\{v_1,...,v_m\}$--indecomposable point $[u]\in \mathbb{R}\mathbb{P}_+^{n-1} $.

\end{theorem}

This theorem clarifies the importance of Farkas--related vectors. In this section,   
we discuss a number of characterizations of Farkas--related vectors.

%%%%%%%%%%%%%%%%%%%%%%%%%%%%%%%%%%%%%%%%%%%%%%%%%%%%%%%%%%%%%%%%%%%%%%%%%%%%%%%%%%%%%%%%%%%%%%%%%%%%%%%%%%%
%%%%%%%%%%%%%%%%%%%%%%%%%%%%%%%%%%%%%%%%%%%%%%%%%%%%%%%%%%%%%%%%%%%%%%%%%%%%%%%%%%%%%%%%%%%%%%%%%%%%%%%%%%%

\begin{subsection}{Characterizations of Farkas--related vectors}

We begin with an easy lemma regarding Farkas-related vectors. 

\begin{lemma} \label{Farkas--related, first characterization}

Let $v_1,...,v_m\in \mathbb{Z}^n$ be arbitrary vectors. Then the following conditions are equivalent.\\
(1) The vectors $v_1,...,v_m\in \mathbb{Z}^n$ are  Farkas--related.\\
(2) If a vector $w\in  \sum_{i=1}^m \mathbb{Z} v_i$ can be written as 
 $w=\sum_{i=1}^m x_i v_i$ for some rational numbers $0\leq  x_1,...,x_m<1$, then 
we have $w=\sum_{i=1}^m y_i v_i$ for some numbers $y_1,...,y_m\in\{0,1\}$, having the property
that for every $i=1,...,m$, if $x_i=0$, then $y_i=0$.  \\
(3) If a vector $w\in \sum_{i=1}^m \mathbb{Z} v_i$ can be written as $k w=\sum_{i=1}^m a_i v_i$ for some integers $0\leq a_1,...,a_m< k$  
(where $k$ is an arbitrary natural number), then 
$w$ can be written as $w=\sum_{i=1}^m y_i v_i$ for some numbers $y_1,...,y_m\in\{0,1\}$, having the property
that for every $i=1,...,m$, if $a_i=0$, then  $y_i=0$.  

\end{lemma}

\begin{proof}

First we prove that (1) implies (2). Suppose that   a vector $w\in  \sum_{i=1}^m \mathbb{Z} v_i$ can be written as 
 $w=\sum_{i=1}^m x_i v_i$ for some rational numbers $0\leq  x_1,...,x_m<1$. Set 
$a_i=b_i=0$ if $x_i=0$ and $a_i=0,b_i=1$ if $x_i\neq 0$. Since the vectors $v_1,...,v_m$ are Farkas--related,
 there exist integers $a_1\leq y_1\leq b_1,...,a_m\leq y_m\leq b_m$ such that  $w=\sum_{i=1}^m y_i v_i$.
It is clear that the numbers $y_1,...,y_m$ satisfy the condition in (2) and we are done.

It is easy to see that (2) and (3) are equivalent, so it remain to prove that (2) implies (1).   
To show this direction, suppose that  for arbitrary integers $a_1\leq b_1,...,a_m\leq b_m$,
a vector $w\in   \sum_{i=1}^m \mathbb{Z} v_i$ can be written as 
 $w=\sum_{i=1}^m x_i v_i$ for some rational numbers $a_1\leq  x_1\leq b_1,...,a_m\leq x_m\leq b_m$.
Then we have 
 $$w-\sum_{i=1}^m[x_i]v_i=\sum_{i=1}^m (x_i-[x_i]) v_i.$$
Since   $0\leq x_i-[x_i]<1$, for $i=1,...,m$, there are integers $y_1,...,y_m\in\{0,1\}$, having the property
that for every $i=1,...,m$, if $x_i-[x_i]=0$, then $y_i=0$, such that 
$w-\sum_{i=1}^m[x_i]v_i=\sum_{i=1}^m y_i  v_i$. Then we have, 
$w=\sum_{i=1}^m([x_i]+y_i) v_i$.
If $[x_i]=b_i$ for some $i\in \{1,...,m\}$, then we have $x_i-[x_i]=0$, which implies that
$y_i=0$ and therefore $a_i\leq [x_i]+y_i\leq b_i$. If $[x_i]<b_i$ for  some $i=1,...,m$, then clearly, we have $a_i\leq [x_i]+y_i\leq b_i$.
So, the presentation $w=\sum_{i=1}^m([x_i]+y_i) v_i$ is the desired one and we are done.   

\end{proof}

Now, we present a useful criterion to check if some given vectors are Farkas-related. Let us introduce some definitions. 
The support of a vector  $v=(x_1,...,x_n)\in \mathbb{Q}^n$ is defined to be $supp(x)=\{i|x_i\neq 0\}$.  
A nonzero vector $v$ in a vector subspace $V$ of $\mathbb{Q}^n$ is called an elementary vector of $V$, 
if $supp(v)$ is minimal with respect to inclusion, in the set $\{supp(w)|0\neq w\in V\}$. An elementary vector
$v=(x_1,...,x_n)$ is called an elementary integral vector if $\{x_i|x_i\neq 0\}$ are relatively prime integers. Note that  
for every nonzero vector $v$ of $V$, there is an elementary integral vector $w$ of $V$ with $supp(w)\subset supp(v)$, see \cite{RO,Vi}.
Given vectors   $v_1,...,v_m\in \mathbb{Q}^n$, we say that a relation $\sum_{i=1}^m a_i v_i=0$ (where $a_1,...,a_m\in \mathbb{Q}$) 
is an elementary (integral) relation if the vector  $(a_1,...,a_m)$ is an elementary (integral) vector of  the vector space 
$\{(x_1,...,x_m)|\sum_{i=1}^m x_i v_i=0\}$. Using this terminology, we can present a useful criterion. 

\begin{proposition} \label{Farkas--related, second characterization}

Let $v_1,...,v_m\in \mathbb{Z}^n$ be arbitrary vectors.  Then the vectors $v_1,...,v_m$ are Farkas--related
if and only if for every elementary integral relation $\sum_{i=1}^m a_i v_i=0$, we have
$a_1,...,a_m\in\{-1,0,1\}$. 

\end{proposition}

\begin{proof}

First we prove the "only if" direction. So suppose that the vectors $v_1,...,v_m$ are Farkas--related 
and let $\sum_{i=1}^m a_i v_i=0$ be an elementary relation. Without loss of generality, we may assume that 
$\{i|a_i\neq 0\}=\{1,...,r\}$. 
We can write
$ v_1=\sum_{i=2}^r \frac{-a_i}{a_1} v_i$. Since, obviously, we have  $v_1\in \sum_{i=1}^m \mathbb{Z} v_i$,
we conclude that there are integers 
$$0\leq y_1\leq 0,[\frac{-a_2}{a_1}]\leq y_{2}\leq  [\frac{-a_2}{a_1}]+1,...,[\frac{-a_r}{a_1}]\leq y_{r}\leq  [\frac{-a_r}{a_1}]+1,$$ 
such that $v_1=\sum_{i= 1}^r y_i v_i$. Since $\sum_{i=1}^m a_i v_i=0$ is an elementary relation, we easily see that 
the vectors $v_2,...,v_r$ are  $\mathbb{Z}$--linearly independent.  Therefore 
we must have $\frac{-a_i}{a_1}=y_i\in \mathbb{Z}$, for $i=2,...,r$, i.e. $a_1$ divides all the numbers 
$a_2,...,a_r$. A similar argument shows that each $a_i$ ($2\leq i\leq r$) divides 
all the numbers $a_1,...,a_r$. Since the numbers $a_1,...,a_r$ are relatively prime, we conclude that 
$a_1,...,a_r\in \{-1,1\}$  and the proof of this direction is complete.  
 
To prove the converse,  suppose that given vectors $v_1,...,v_m\in \mathbb{Z}^n$ satisfy  the condition. 
First, we show that 
if $k w\in  \sum_{i\in I} \mathbb{Z} v_i$ for some vector $w\in   \sum_{i=1}^m \mathbb{Z} v_i$, some 
nonempty set $I\subset \{1,...,m\}$ and 
some nonzero integer $k$, then we have $w\in  \sum_{i\in I} \mathbb{Z} v_i$. To show this,  
we use induction on $m-|I|$. There is nothing to prove in the base case, i.e. $m-|I|=0$. 
To prove the inductive step, without loss of generality, we may assume that $m\notin I$. Then, by induction,we have
$w\in  \sum_{i\in I\cup \{m\}} \mathbb{Z} v_i$, i.e.
$w=\sum_{i\in I\cup \{m\}} b_i v_i$, for some integers $b_i$ ($i\in I\cup \{m\}$). If $b_m= 0$, then
we are done. So suppose that  $b_m\neq 0$. Then we have $k b_m v_m\in \sum_{i\in I} \mathbb{Z} v_i$. It follows that, 
there is a nonempty set $J\subset I$, such that the vectors $\{v_j\}_{j\in J}$, are $\mathbb{Z}$--linearly independent,
and  $k b_m v_m\in  \sum_{i\in J} \mathbb{Z} v_i$. So, we must have 
$\sum_{i\in J\cup \{m\}} a_i v_i=0$, for some integers $a_i\in\{-1,0,1\}$, not all equal to zero. Since  
the vectors $\{v_j\}_{j\in J}$, are $\mathbb{Z}$--linearly independent, we have $a_m\neq 0$. 
Since $a_m\in  \{-1,1\}$, we conclude that $v_m\in \sum_{i\in I} \mathbb{Z} v_i$ and therefore, $w\in \sum_{i\in I} \mathbb{Z} v_i$.

To show that the vectors $v_1,...,v_m$ are  Farkas--related, 
we use Lemma \ref{Farkas--related, first characterization}, part (3).
So let $k w=\sum_{i=1}^m b_i v_i$ 
for some vector $w\in \sum_{i=1 }^m \mathbb{Z} v_i$, and some integers $0\leq b_1,...,b_m<k$, where $k$ is a natural number.
We need to prove that   $w=\sum_{i=1}^m y_i v_i$ for some numbers  $y_1,...,y_m\in \{0,1\}$, having the property
that for every $i=1,...,m$, if $b_i=0$, then $y_i=0$.  
To do so, we use induction on $m$. First suppose that $m=1$. Since $w\in \mathbb{Z} v_1$, we have $w=l v_1$ for some integer $l$. 
Then we have $k l v_1=k w=a_1 v_1$. Since $0\leq a_1<k$, this identity is possible, only if $w=0$, in which case we have $w=0 \times v_1$, and 
we are done.

Now we prove the inductive step. If the vectors $v_1,...,v_m$ are  $\mathbb{Z}$-linearly independent, then from 
$w=\sum_{i=1}^m \frac{b_i}{k} v_i$ and the facts that  $w\in  \sum_{i= 1}^m \mathbb{Z} v_i$ and $0\leq b_i <k$ for $i=1,...,m$, 
we conclude that $b_1=\cdots=b_m=0$ and therefore $w=0$ and we are done.

If some $b_i$, say $b_1$, is zero, then from $kw=\sum_{i=2}^m b_i v_i$, we conclude that $w\in \sum_{i= 2}^m \mathbb{Z} v_i$. 
Clearly, the vectors $v_2,...,v_m$ satisfy the condition of the proposition and therefore by induction, we have 
$w=\sum_{i=2}^m y_i v_i$ for some numbers  $y_2,...,y_m\in \{0,1\}$, having the property
that for every $i=2,...,m$, if $b_i=0$, then $y_i=0$.  Then the numbers $y_1=0,y_2,...,y_m$ give us the desired presentation.

So we may suppose that $v_1,...,v_m$ are  $\mathbb{Z}$-linearly dependent and none of the numbers $b_1,...,b_m$ are zero. 
Then there exists an  elementary integral relation $\sum_{i=1}^m a_i v_i=0$. 
Since $a_1,...,a_m\in \{-1,0,1\}$, it is then easy to see that 
we can choose an integer $ l>0$, large enough, such that for each $i=1,...,m$, we have $0\leq b_i+l a_i\leq k$ and at least of  the numbers 
$b_1+la_1,...,b_{m}+l a_{m}$ is equal to $0$ or $k$.  
Without loss of generality, we may assume that 
$$0=b_1+l a_1=\cdots=b_r+l a_r,$$
$$0 <b_{r+1}+l a_{r+1},...,b_{s}+l a_{s}<k,$$
$$b_{s+1}+l a_{s+1}=\cdots= b_{m}+l a_{m}=k.$$ 
So, we can write
\begin{equation}\label{w}
k(w-v_{s+1}-\cdots-v_m)=(b_{r+1}+l a_{r+1})v_{r+1}+\cdots+(b_{s}+l a_{s})v_{s}.
\end{equation}
Clearly, we have $0\leq s-r < m$. If $s-r=0$, then by Equality \ref{w}, we have $w=\sum_{i=s+1}^m v_{i}$, and we are done.
So suppose that $s-r>0$. Clearly, the vectors $v_{r+1},...,v_s$ satisfy the condition of the proposition. We have seen that, 
Equality \ref{w}, implies that $w-\sum_{i=s+1}^m v_{i}\in \sum_{i=r+1}^s\mathbb{Z} v_i$, because  $w\in \sum_{i=1}^m\mathbb{Z} v_i$.
Therefore, by induction, we have $w-\sum_{i=s+1}^m v_{i}=\sum_{i=1}^{s-r} c_{i}v_{r+i}$ for some integers 
$c_1,...,c_{s-r}\in\{0,1\}$. So, we have $w=\sum_{i=s+1}^m v_{i}+\sum_{i=1}^{s-r} c_{i}v_{r+i}$ and hence the proof is complete.   

\end{proof}

An immediate consequence of the above proposition is that if distinct vectors $v_1,...,v_m\in\mathbb{Z}^n\setminus\{0\}$ are Farkas--related vectors, then $m< 3^n$.  It would be interesting to 
determine the maximum number of distinct nonzero Farkas--related vectors in $\mathbb{Z}^n$ (and possibly classify such "maximal" sets of vectors). 

In the end of this section, we  introduce a construction, producing new Farkas-related vectors from a given set of Farkas-related vectors. 
Let us call an integral matrix, a \textit{Farkas matrix} if its columns are Farkas--related vectors. 
Our construction in terms of matrices, is the following.

\begin{proposition}\label{construction}

Let $A,B$ be two $n\times m$ integral matrices, $C$ be an invertible $m\times m$ integral matrix, and $D$ be  an $m\times m$ matrix, having at most one nonzero entry, equal to 
$1$ or $-1$, in each row. Then the matrix
$$ E=\begin{pmatrix}
A & B\\
CD & C
\end{pmatrix} $$
is a Farkas matrix if and only if the matrix $A-B D$ is a Farkas matrix.

\end{proposition}

\begin{proof}

Let $EX=0$, where 
$X=\begin{pmatrix}
x_1\\
\vdots\\
x_{2 m}
\end{pmatrix} $ is a vector in   the null space of $E$. Setting 
$Y=\begin{pmatrix}
x_1\\
\vdots\\
x_{m}
\end{pmatrix} $ and $Z=\begin{pmatrix}
x_{m+1}\\
\vdots\\
x_{2 m}
\end{pmatrix} $, one can easily see that $EX=0$ if and only if $A Y+B Z=0, C D Y+CZ=0$. Since $C$ is invertible, 
 these equations are equivalent to the equations $(A-B D)Y=0, Z=-D Y$. Now the proof can be completed by using 
Proposition \ref{Farkas--related, second characterization} and the following 
easily verifiable facts.  We have $x_i\in\{-1,0,1\}$ ($1\leq i \leq 2 m$), if and only if  $x_i\in\{-1,0,1\}$ ($1\leq i \leq m$). 
The vector $X$ is an elementary integral vector in the null space of $E$ if and only if the vector $Y$ is 
 an elementary integral vector in the null space of $A-B D$. 

\end{proof}

\end{subsection}

%%%%%%%%%%%%%%%%%%%%%%%%%%%%%%%%%%%%%%%%%%%%%%%%%%%%%%%%%%%%%%%%%%%%%%%%%%%%%%%%%%%%%%%%%%%%%%%%%%%%%%%%%%%
%%%%%%%%%%%%%%%%%%%%%%%%%%%%%%%%%%%%%%%%%%%%%%%%%%%%%%%%%%%%%%%%%%%%%%%%%%%%%%%%%%%%%%%%%%%%%%%%%%%%%%%%%%%

\end{section}

%%%%%%%%%%%%%%%%%%%%%%%%%%%%%%%%%%%%%%%%%%%%%%%%%%%%%%%%%%%%%%%%%%%%%%%%%%%%%%%%%%%%%%%%%%%%%%%%%%%%%%%%%%%
%%%%%%%%%%%%%%%%%%%%%%%%%%%%%%%%%%%%%%%%%%%%%%%%%%%%%%%%%%%%%%%%%%%%%%%%%%%%%%%%%%%%%%%%%%%%%%%%%%%%%%%%%%%
%%%%%%%%%%%%%%%%%%%%%%%%%%%%%%%%%%%%%%%%%%%%%%%%%%%%%%%%%%%%%%%%%%%%%%%%%%%%%%%%%%%%%%%%%%%%%%%%%%%%%%%%%%%
%%%%%%%%%%%%%%%%%%%%%%%%%%%%%%%%%%%%%%%%%%%%%%%%%%%%%%%%%%%%%%%%%%%%%%%%%%%%%%%%%%%%%%%%%%%%%%%%%%%%%%%%%%%
%%%%%%%%%%%%%%%%%%%%%%%%%%%%%%%%%%%%%%%%%%%%%%%%%%%%%%%%%%%%%%%%%%%%%%%%%%%%%%%%%%%%%%%%%%%%%%%%%%%%%%%%%%%
%%%%%%%%%%%%%%%%%%%%%%%%%%%%%%%%%%%%%%%%%%%%%%%%%%%%%%%%%%%%%%%%%%%%%%%%%%%%%%%%%%%%%%%%%%%%%%%%%%%%%%%%%%%
%%%%%%%%%%%%%%%%%%%%%%%%%%%%%%%%%%%%%%%%%%%%%%%%%%%%%%%%%%%%%%%%%%%%%%%%%%%%%%%%%%%%%%%%%%%%%%%%%%%%%%%%%%%
%%%%%%%%%%%%%%%%%%%%%%%%%%%%%%%%%%%%%%%%%%%%%%%%%%%%%%%%%%%%%%%%%%%%%%%%%%%%%%%%%%%%%%%%%%%%%%%%%%%%%%%%%%%

\begin{section}{Farkas-related vectors in Graph Theory}

In this section a class of examples on Farkas--related vectors appearing in Graph Theory, is presented.  
We follow the terminology of \cite{BM}, except that here, the word graph means simple graph. 

\begin{subsection}{Incidence matrices of graphs} 

The question, considered in this part, is the following: When are the columns of the incidence matrix of a graph Farkas--related?
Let us introduce some notations. Suppose that $G$ is a graph with $V(G)=\{1,...,n\}$ and  $E(G)=\{e_1,...,e_m\}$.  
Let $M=M(G)$  denote the incidence matrix 
of $G$, i.e. $M$ is an $n\times m$ matrix with $M_{i j}=1$ if the vertex $i$ is an end of the edge $e_j$ and $M_{i j}=0$, otherwise.  
Furthermore, let $f_1,...,f_n$ denote the standard  basis of $\mathbb{Z}^n$. 
It is clear that the column of $M(G)$ corresponding to the edge $e=i j$ is  the vector $v(e)=v_G(e)=f_{i}+f_{j}$.
An indecomposable point with respect to  the columns of $M(G)$ is, for simplicity, called a $G$-indecomposable point. 
First we characterize the $G$-indecomposable points for a given graph $G$.

\begin{lemma}\label{G-indecomposable}

Suppose that $G$ is a connected graph with $V(G)=\{1,...,n\}$. Then we have the following.\\
(I) If  $G$ is a bipartite graph with bipartition $(\{1,...,r\},\{r+1,...,n\})$, then the set of the $G$-indecomposable points consists of the points
$$[\pm u_{I,J}]=[ \pm \frac{|I|+|J|}{n}(-\sum_{i=1}^r f_i+\sum_{i=r+1}^{n} f_{i})+\sum_{i\in I} f_i-\sum_{j\in J} f_j],$$
where $I\subset \{1,...,r\}$ and  $J\subset \{r+1,...,n\}$
are two sets such that the induced subgraphs of $G$ on $I\cup J$ and on $(\{1,...r\}\setminus I)\cup (\{r+1,...,n\}\setminus J)$, are connected. \\
(II) If $G$ is not a bipartite graph, then the set of the $G$-indecomposable points consists of the points 
$[\sum_{i\in I} f_i-\sum_{j\in J} f_j]$,   where $I, J\subset \{1,...,n\}$  are disjoint sets such that 
the subgraph of $G$ whose set  of vertices is $I\cup J$ and set of edges  is $\{i j\in E(G)| i\in I, j\in J\}$, is a connected graph and
no connected component  of the subgraph $G-(I\cup J)$ is a bipartite graph.   

\end{lemma}

\begin{proof}

According to Lemma \ref{indecomposable, characterization}, we need to find all sets $K\subset E(G)$ for which the vector space
$$V(K)=\{w\in \sum_{e\in E(G)} \mathbb{R} v(e)| (w,v(e))=0\quad \text{for all}\quad e\in K\},$$
is one dimensional. Let  $H$ be the subgraph of $G$ with $V(H)=V(G)$ and $E(H)= K$. Suppose that a vector $w=\sum_{i=1}^n a_i f_i\in    \mathbb{R}^n$ belongs to $V(K)$.
It is easy to verify the following. If there is a path of even (odd) length from vertex $i$ to vertex $j$ in $H$, then we have $a_i=a_j$  ($a_i=-a_j$). In
particular, if there is a cycle of odd length, containing a vertex $i$ in $G$, then we have $a_i=0$.

Now, let $G_1,...,G_k$
be the connected components of $H$, containing a cycle of odd length and let $H_1,...,H_l$
be the connected components of $H$, containing no cycles of odd length. It follows that 
the graphs $H_1,...,H_l$ are bipartite graphs. Denote the set of vertices of 
$G_i$ by $I_i$, and choose a bipartition   $(J_i,K_i)$ for $H_i$. From the above discussion, it is easy to see that  
the vector space  $V(K)$ consists of all vectors 
$$w=\sum_{i=1}^l (r_i (\sum_{j\in J_i} f_j-\sum_{j\in K_i} f_j)),$$ 
where $r_1,...,r_l\in\mathbb{R}$ are arbitrary real numbers, such that $w\in \sum_{e\in E(G)} \mathbb{R} v(e)$. 
Now, we treat two parts of the lemma separately. 
\newline
(I) $G$ is a bipartite graph with bipartition $(\{1,...,r\},\{r+1,...,n\})$. In this case, It is known that  the rank of $M(G)$ is $n-1$.
It is then easy to see that  a vector $w=\sum_{i=1}^n a_i f_i\in    \mathbb{R}^n$
belongs to  $\sum_{e\in E(G)} \mathbb{R} v(e)$ if and only if $\sum_{i=1}^r a_i=\sum_{i=r+1}^n a_i$.  So   
the dimension of $V(K)$ is equal to $l-1$.
Therefore the vector space $V(K)$ is one dimensional if and only if $l=2$.
Let $l=2$, i.e. $J_1\cup J_2=\{1,...,r\}$ and $K_1\cup K_2=\{r+1,...,n\}$, because $k=0$. Using the above description of vectors in $V(K)$, we see that 
the vector 
$$u_{J_1,K_1}=\frac{|J_1|+|K_1|}{n}(-\sum_{i=1}^r f_i+\sum_{i=r+1}^{n} f_{i})+\sum_{i\in J_1} f_i-\sum_{j\in K_1} f_j,$$
forms a basis for   the vector space $V(K)$, and we are done. 
\newline 
(II) $G$ is not a bipartite graph. In this case, it is known that the rank of $M(G)$ is $n$, or equivalently,  $\sum_{e\in E(G)} \mathbb{R} v(e)=\mathbb{R}^n$.
So, the dimension of $V(K)$ is equal to 
$l$. Therefore the vector space $V(K)$ is one dimensional if and only if $l=1$. Let $l=1$, i.e.  
$J_1\cup K_1=\{1,...,n\}$. Using the above description of vectors in $V(K)$, we see that 
the vector $\sum_{i\in J_1} f_i-\sum_{j\in K_1} f_j$,
forms a basis for   the vector space $V(K)$, and we are done.  

\end{proof}

Now, we answer the question raised in the beginning of this part.

\begin{proposition}\label{Farkas graph}

The  incidence matrix of a graph $G$ is a  Farkas matrix
if and only if $G$ does not have two edge-disjoint cycles of odd lengths, connected by a path.  

\end{proposition}

\begin{proof}

By proposition \ref{Farkas--related, second characterization}, 
the incidence matrix of a graph $G$ is a  Farkas matrix
if and only if for every elementary integral vector $(a_1,...,a_m)$ in the null space of the incidence matrix, we have 
$a_1,...,a_m\in \{-1,0,1\}$. It is known that the last statement is equivalent to the statement that 
$G$ does not have two edge-disjoint cycles of odd lengths, connected by a path, see Proposition 4.2 and Corollary 4.1 in \cite{Vi}.

\end{proof}

In particular, we can apply Theorem \ref{Farkas, integers, first}  to the incidence matrix of the graphs satisfying the condition in the above proposition,
e.g. bipartite graphs. This leads to the following theorem.

\begin{theorem}  \label{GALERYSER}

Let $G$ be  a bipartite graph with bipartition $(\{1,...,r\},\{r+1,...,n\})$. 
Let $s_i, a_{e}\leq b_{e}$ be integers where $i=1,...,n$, and $e\in E(G)$. Then there exist integers 
$a_{e}\leq x_{e}\leq b_{e}$ ($e\in E(G)$), such that 
$$\sum_{i=1}^n s_i f_i=\sum_{e\in E(G)} x_{e} v(e),$$ 
if and only if the following conditions hold:\\
(1) $s_1+\cdots+s_r=s_{r+1}+\cdots+s_{n}.$\\
(2) For all sets $I\subset \{1,...,r\}, J\subset  \{r+1,...,n\}$, such that
the induced subgraphs of $G$ on $I\cup J$ and on $(\{1,...r\}\setminus I)\cup (\{r+1,...,n\}\setminus J)$ are connected, we have
$$\sum_{i\in I}s_i-\sum_{j\in J} s_{j}\leq  \sum_{e=i j, i\in I,j\notin J} b_{e}-\sum_{e= i j, i\notin I,j\in J} a_{e}.$$
% and $$\sum_{j\in J} s_{j}-\sum_{i\in I}s_i\leq  \sum_{e= i j, i\notin I,j\in J} b_{e}- \sum_{e=i j, i\in I,j\notin J} a_{e}.$$

\end{theorem}

\begin{proof}

Using Theorem \ref{Farkas, integers, first}, Lemma \ref{G-indecomposable} and Proposition \ref{Farkas graph},  we see that there exist integers 
$a_{e}\leq x_{e}\leq b_{e}$ ($e\in E(G)$) such that 
$$\sum_{i=1}^n s_i f_i=\sum_{e\in E(G)} x_{e} v(e),$$ 
if and only if the following conditions hold:\\
(1) $\sum_{i=1}^n s_i f_i\in\sum_{e\in E(G)} \mathbb{Z} v(e)$,
 which one can easily see that, is equivalent to the identity $s_1+\cdots+s_r=s_{r+1}+\cdots+s_{n}$.\\
(2)  For all sets $I\subset \{1,...,r\},J\subset  \{r+1,...,n\}$, such that 
the induced subgraphs of $G$ on $I\cup J$ and on $(\{1,...r\}\setminus I)\cup (\{r+1,...,n\}\setminus J)$ are connected, we have
$$(\pm u_{I J},\sum_{i=1}^n s_i f_i)\leq \sum_{i j\in E(G)} a_{i j}\frac{(\pm u_{I J},f_i+f_j)-|( \pm u_{I J},f_i+f_j)|}{2}$$
$$+ \sum_{i j\in E(G)}  b_{i j}\frac{(\pm u_{I J},f_i+f_j)+|(\pm  u_{I  J},f_i+f_j)|}{2}.$$
This inequality, in view of (1), is easily simplified and we obtain the desired inequalities.

\end{proof}

Note that, in the special case where $G$ is the complete bipartite graph, the above theorem gives us the well-known Gale-Ryser theorem, 
see \cite{BR} for a general discussion on this theorem and related topics. 
Similarly, we obtain the following result. 

\begin{theorem}  \label{GALERYSER 2}

Let $G$ be  a  graph with $V(G)=\{1,...,n\}$, satisfying the condition in Proposition  \ref{Farkas graph}. In addition, assume that $G$ is not a bipartite graph. 
Let $s_i, a_{e}\leq b_{e}$ be integers where $i=1,...,n$, and $e\in E(G)$. Then there exist integers 
$a_{e}\leq x_{e}\leq b_{e}$ ($e\in E(G)$), such that 
$$\sum_{i=1}^n s_i f_i=\sum_{e\in E(G)} x_{e} v(e),$$ 
if and only if $\sum_{i=1}^n s_i  $ is even and for all disjoint sets $I,J  \subset \{1,...,n\}$,
such that the subgraph of $G$, whose set  of vertices is $I\cup J$ and set of edges  is $\{i j\in E(G)| i\in I, j\in J\}$, is a connected graph and
no connected component of the subgraph $G-(I\cup J)$ is a bipartite graph, we have
$$\sum_{i\in I}s_i-\sum_{j\in J} s_{j}\leq  \sum_{e=i j, i\in I,j\notin I\cup J} b_{e}+ 2 \sum_{e=i j, i,j\in I} b_{e}-\sum_{e= i j, i\notin I\cup J,j\in J} a_{e}-2\sum_{e= i j, i,j\in J} a_{e}.$$

\end{theorem}

\begin{proof}

The proof is similar to the proof of Theorem \ref{GALERYSER}. The only point deserving some explanation is that  
a vector $\sum_{i=1}^n s_i f_i\in\mathbb{Z}^n$ belongs to $ \sum_{e\in E(G)} \mathbb{Z} v(e)$ if and only if
$\sum_{i=1}^n s_i f_i $ is even. In fact this can be proved inductively for any graph containing a cycle of odd length. 
One first proves this for a cycle of odd length and then uses an induction on $m+n$ to prove it for the general case. 

\end{proof}

\end{subsection}

%%%%%%%%%%%%%%%%%%%%%%%%%%%%%%%%%%%%%%%%%%%%%%%%%%%%%%%%%%%%%%%%%%%%%%%%%%%%%%%%%%%%%%%%%%%%%%%%%%%%%%%%%%%
%%%%%%%%%%%%%%%%%%%%%%%%%%%%%%%%%%%%%%%%%%%%%%%%%%%%%%%%%%%%%%%%%%%%%%%%%%%%%%%%%%%%%%%%%%%%%%%%%%%%%%%%%%%

\begin{subsection}{Incidence matrices of oriented graphs} 

We define an oriented graph to be a directed graph with no loops and no multiple arcs. 
Suppose that $D$ is an oriented graph with $V(D)=\{1,...,n\}$ and $ A(D)=\{e_1,...,e_m\}$.   
The directed incidence matrix  $N=N(D)$ 
of $D$ is defined to be the following matrix: $N$ is an $n\times m$ matrix with 
\[ N_{i j} = \left\{ 
  \begin{array}{l l}
1 & \quad \text{if vertex $i$ is the tail of $e_j$}\\
-1 & \quad \text{if vertex $i$ is the head of $e_j$}\\
0 & \quad \text{otherwise}\\
 \end{array} \right.\]    
Let $f_1,...,f_n$ denote the standard  basis of $\mathbb{Z}^n$. 
It is clear that the column of $N$ corresponding to the arc $e=\overrightarrow{i j}$ is  the vector $v(e)=v_D(e)=f_{i}-f_{j}$.
An indecomposable point with respect to  the columns of $N$ is, for simplicity, called a $D$-indecomposable point. 
First we characterize the $D$-indecomposable points for a given oriented graph $D$. An oriented graph $D$ is called 
connected if its underlying undirected graph is connected. 

\begin{lemma}\label{D-indecomposable}

Suppose that $D$ is a connected oriented graph with $V(D)=\{1,...,n\}$. Then   the set of the  $D$-indecomposable points consists of  the points
$[u_{I}]=[ n\sum_{i\in I} f_i-|I|\sum_{i=1}^n f_i],$
where $\emptyset \neq I\subsetneq \{1,...,n\}$ 
is a set such that the induced subgraphs of $D$ on $I$ and $\{1,...,n\}\setminus I$ are connected.

\end{lemma}

\begin{proof}

According to Lemma \ref{indecomposable, characterization}, we need to find all sets $K\subset A(D)$ for which the vector space
$$V(K)=\{w\in \sum_{e\in A(D)} \mathbb{R} v(e)| (w,v(e))=0\quad \text{for all}\quad e\in K\},$$
is one dimensional. Let  $H$ be the undirected graph with $V(H)=V(D)$ and $E(H)= K$. Suppose that a vector $w=\sum_{i=1}^n a_i f_i\in    \mathbb{R}^n$ belongs to $V(K)$.
It is easy to see that  if there is a path from vertex $i$ to vertex $j$ in $H$, then we must have $a_i=a_j$. 

Now,   let $H_1,...,H_l$
be the connected components of  $H$.  Denote the set of vertices of 
$H_i$ by $I_i$. From the above discussion, it is easy to see that  
the vector space  $V(K)$ consists of all vectors 
$$w=\sum_{i=1}^l (r_i \sum_{j\in I_i} f_j),$$ 
where $r_1,...,r_l\in\mathbb{R}$ are arbitrary real numbers, such that $w\in \sum_{e\in A(D)} \mathbb{R} v(e)$. 
But, it is easy to prove (by induction on $m+n$ for example) that a vector $\sum_{i=1}^n a_i f_i\in    \mathbb{R}^n$
belongs to  $\sum_{e\in A(D)} \mathbb{R} v(e)$ if and only if $\sum_{i=1}^n a_i=0$. So 
$$V(K)=\{\sum_{i=1}^l (r_i \sum_{j\in I_i} f_j)|r_1,...,r_l\in\mathbb{R},  \sum_{i=1}^l r_i|I_i|=0\}.$$ 
In particular, the vector space $V(K)$ is one dimensional if and only if $l=2$.
Let $l=2$, i.e. $I_1\cup I_2=\{1,...,n\}$. Using the above description of vectors in $V(K)$, we see that 
the vector 
$$u_{I_1}=n\sum_{i\in I_1} f_i-|I_1|\sum_{i=1}^n f_i,$$
forms a basis for the vector space $V(K)$, and we are done.

\end{proof}

For oriented graphs, we have the following result.  

\begin{proposition}\label{oriented Farkas graph}

 For arbitrary oriented graph $D$, the directed incidence matrix of $D$ is a Farkas matrix.  

\end{proposition}

\begin{proof}

By proposition \ref{Farkas--related, second characterization}, 
  the directed incidence matrix of $D$ is a Farkas matrix
if and only if for every elementary integral vector $(a_1,...,a_m)$ in the null space of $N(D)$, we have 
$a_1,...,a_m\in \{-1,0,1\}$, which is a known fact, see page 204 of \cite{RO1}. In fact one can easily show that an integral elementary vector in   
 the null space of $N(D)$ corresponds to a directed cycle in $D$.

\end{proof}

In particular, we can apply Theorem \ref{Farkas, integers, first}  to the directed incidence matrix of an oriented graph.   
This leads to the following theorem.

\begin{theorem}  \label{Landua}

Let $D$ be a connected oriented graph  with $V(D)=\{1,...,n\}$ and let $r_i, a_{e}\leq b_{e}$ be integers where $i=1,...,n$ and $e\in A(D)$. Then there exist integers 
$a_{e}\leq x_{ e}\leq b_{e }$ ($e\in A(D)$) such that 
$$\sum_{i=1}^n r_i f_i=\sum_{e\in A(D)} x_{e} w(e),$$ 
if and only if the following conditions hold:\\
(1) $r_1+\cdots+r_n=0.$\\
(2) For all sets $\emptyset \neq I\subsetneq  \{1,...,n\}$, such that the induced subgraphs of $D$ on $I$ and $\{1,...,n\}\setminus I$ are connected, 
we have
$$\sum_{i\in I} r_i\leq  \sum_{e=\overrightarrow{i j},i\in I,j\notin I} b_{e}- \sum_{e=\overrightarrow{i j}, i\notin I,j\in I} a_{e}.$$

\end{theorem}

\begin{proof}

The proof  is similar to the proof of Theorem \ref{GALERYSER}. 
The only point deserving some explanation is the following: The vector $\sum_{i=1}^n r_i f_i\in \mathbb{Z}^n$ belongs to $\sum_{e\in A(D)} \mathbb{Z} w(e)$, if and only if
$r_1+\cdots+r_n=0.$ This statement can easily be proved by induction on $m+n$.

\end{proof}

We can use the above theorem to derive a result concerning "signed graphical sequences" as follows. Let $D$ be  an oriented graph with  $V(D)=\{1,...,n\}$.
For every vertex $i$, denote the outdegree and the indegree of $i$ by $d^+(i)$ and $d^-(i)$. The total degree of $i$ is defined by $d^0(i)=d^+(i)-d^-(i)$.
The sequence $(d^0(1),...,d^0(n))$ is called the \textit{signed degree sequence} of $D$. A sequence $(d_1,...,d_n)$ of integers is called a \text{signed graphical sequence} if there is an 
oriented graph on vertices $1,...,n$ such that $d_1=d^0(1),...,d_n=d^0(n)$. Now, we have the following result characterizing signed graphical sequences, see also \cite{Av}.

\begin{corollary}\label{oriented graphical sequences}

A nonincreasing  sequence $(d_1,...,d_n)$ of integers is a signed graphical sequence if and only if the following conditions hold:\\
(1) $d_1+\cdots+d_n=0.$\\
(2) For all natural numbers $1\leq l \leq n$,  we have $\sum_{i= 1}^l d_i\leq l(n-l).$

\end{corollary}  

\begin{proof}

Let $D$ be the following oriented graph. The set of vertices of $D$ is $V(D)=\{1,...,n\}$, and 
the  set of arcs of $D$ is $A(D)=\{\overrightarrow{i j}|1\leq i< j\leq n\}$. Then it is easy to see that 
a sequence $(d_1,...,d_n)$ of integers is a signed graphical sequence if and only if there exist integers $x_e\in\{-1,0,1\}$ ($e\in A(D)$), such that 
$\sum_{i=1}^n d_i f_i=\sum_{e\in A(D)} x_{e} w(e)$. Setting, $a_e=-b_e=-1$ ($e\in A(D)$) in Theorem \ref{Landua}, one can easily finish the proof.

\end{proof}

\end{subsection}

%%%%%%%%%%%%%%%%%%%%%%%%%%%%%%%%%%%%%%%%%%%%%%%%%%%%%%%%%%%%%%%%%%%%%%%%%%%%%%%%%%%%%%%%%%%%%%%%%%%%%%%%%%%
%%%%%%%%%%%%%%%%%%%%%%%%%%%%%%%%%%%%%%%%%%%%%%%%%%%%%%%%%%%%%%%%%%%%%%%%%%%%%%%%%%%%%%%%%%%%%%%%%%%%%%%%%%%

\begin{subsection}{Orientations on graphs} 

Suppose that $G$ 
is a graph  with $V(G)=\{1,...,n\}$ and $E(G)=\{e_1,...,e_m\}$. Let $f_1,...,f_n,g_1,...,g_m$ be the standard basis for 
$\mathbb{Z}^n\oplus  \mathbb{Z}^m$. For each edge $e=e_k=i j\in E(G)$ ($i<j$), we set
$z(e)=f_i-f_j+g_k$ and $z'(e)=-f_i+f_j+g_k$. 
An indecomposable point with respect to  the vectors $z(e),z'(e)$  ($e\in E(G)$) is, for simplicity, called a $(G)$-indecomposable point. 
For a set $I\subset \{1,...,n\}$, we denote the set of edges with only one end in $I$ by $E(I)$.
First we characterize the $(G)$-indecomposable points for a given graph $G$. 

\begin{lemma}\label{(G)-indecomposable}

Suppose that $G$ is a connected graph with $V(G)=\{1,...,n\}$, as above. Then   the set of $(G)$-indecomposable points consists of  (1) the points
$$[ n\sum_{i\in I} f_i-|I|\sum_{i=1}^n f_i+n(\sum_{e_k\in E(I)\setminus J}g_k-\sum_{e_k\in J}g_k)],$$
with $J\subset E(I)$, where $\emptyset \neq I\subsetneq \{1,...,n\}$ 
is a set such that the induced subgraphs of $G$ on $I$ and $\{1,...,n\}\setminus I$ are connected, and (2) the points $[\pm g_k]$, where 
$e_k\in E(G)$ is an edge such that $G-{e_k}$ is connected.

\end{lemma}

\begin{proof}

According to Lemma \ref{indecomposable, characterization}, we need to find all sets $K, K'\subset E(G)$ for which the vector space
$$V(K, K')=$$
$$\{w\in \sum_{e\in E(G)} \mathbb{R} z(e)+\sum_{e\in E(G)} \mathbb{R} z'(e)| (w,z(e))=(w,z'(e'))=0\,\, \text{for all}\,\, e\in K, e'\in K'\},$$
is one dimensional. Let  $H$ be the subgraph of $G$ with $V(H)=V(G)$ and $E(H)= K\cap K'$. Suppose that a vector 
$w=\sum_{i=1}^n a_i f_i+\sum_{j=1}^m b_j g_j\in    \mathbb{R}^n\oplus  \mathbb{R}^m$ belongs to $V(K, K')$.
It is easy to see that if there is a path from vertex $i$ to vertex $j$ in $H$, then we have $a_i=a_j$. Moreover
 for every edge $e_k\in K\cap K'$ we have $b_k=0$; for every
edge $e_k=i j\in K\setminus K'$ ($i<j$) we have $b_k=a_j-a_i$; and for every 
edge $e_k=i j\in K'\setminus K$ ($i<j$) we have $b_k=a_i-a_j$. 

Now,   let $H_1,...,H_l$
be the connected components of the graph $H$.  Denote the set of vertices of 
$H_i$ by $I_i$. An edge $e\in E(G)$ is said to be of type $(i,j)$ (where $i< j$) if one of its  vertices belong to $I_i$ and 
the other one belongs to $I_j$. From the above discussion, it is easy to see that  
the vector space  $V(K, K')$ consists of all vectors 
$$w=\sum_{i=1}^l (r_i \sum_{j\in I_i} f_j)+\sum_{e_k\in K\setminus K'\, \text{ is of type (i,j)}} (r_j-r_i)g_k$$
$$+\sum_{e_k\in K'\setminus K\, \text{is of type (i,j)}} (r_i-r_j)g_k+\sum_{e_k\notin K\cup K'} s_k g_k ,$$ 
where $r_1,...,r_l,s_k\in\mathbb{R}$ (for $k$ with $e_k\notin K\cup K'$) are arbitrary real numbers, 
such that $w\in \sum_{e\in E(G)} \mathbb{R} z(e)+\sum_{e\in E(G)} \mathbb{R} z'(e)$. 
But, since we have
$$\sum_{e\in E(G)}\mathbb{R} z(e)+\sum_{e\in E(G)} \mathbb{R} z'(e)=\sum_{i j\in E(G)}\mathbb{R}(f_i-f_j)+\sum_{k=1}^m \mathbb{R} g_k,$$
it is easy to prove (by induction on $m+n$ for example) that a 
vector $\sum_{i=1}^n a_i f_i+\sum_{j=1}^m b_j g_j\in    \mathbb{R}^n\oplus  \mathbb{R}^m$
belongs to  $ \sum_{e\in E(G)} \mathbb{R} z(e)+\sum_{e\in E(G)} \mathbb{R} z'(e)$
 if and only if $\sum_{i=1}^n a_i=0$. So, the vector space $V(K, K')$ consists of all vectors 
$$w=\sum_{i=1}^l (r_i \sum_{j\in I_i} f_j)+\sum_{e_k\in K\setminus K'\, \text{ is of type (i,j)}} (r_j-r_i)g_k$$
$$+\sum_{e_k\in K'\setminus K\, \text{is of type (i,j)}} (r_i-r_j)g_k+ \sum_{e_k\notin K\cup K'} s_k g_k,$$ 
where $r_1,...,r_l,s_k\in\mathbb{R}$ (for $k$ with $e_k\notin K\cup K'$) are arbitrary real numbers, such that 
$ \sum_{i=1}^l r_i|I_i|=0$. In particular, the vector space $V(K, K')$
 is one dimensional if and only if $l+|\{k|e_k\notin K\cup K'\}|=2$ if and only if
$l=2, |\{k|e_k\notin K\cup K'\}|=0$, or $l=|\{k|e_k\notin K\cup K'\}|=1$. In the first case, we have
$$I_1\cup I_2=\{1,...,n\}, E(G)=E(H_1)\cup E(H_2)\cup (K\setminus K')\cup (K'\setminus K).$$ 
Setting $I=I_1$, $J=K\setminus K'$ and using the above description of vectors in $V(K)$, we see that 
the vector 
$$u_{I,J}=n\sum_{i\in I} f_i-|I|\sum_{i=1}^n f_i+n(\sum_{e_k\in E(I)\setminus J}g_k-\sum_{e_k\in J}g_k),$$
forms a basis for the vector space $V(K)$. This gives us the $(G)$-decomposable points of form (1). 
In the second case, there is a unique edge $e_k\notin K\cup K'$ such that
$G- e_k$ is connected and the corresponding $(G)$-indecomposable point is $[\pm g_k]$.

\end{proof}

Next, we show that the vectors $z(e),z'(e)$  ($e\in E(G)$) are Farkas--related. 

\begin{proposition}\label{directed Farkas graph 1}

 For every connected graph $G$,  the vectors $z(e),z'(e)$  ($e\in E(G)$) are Farkas--related.  

\end{proposition}

\begin{proof}

Consider the following orientation $D_0$ on $G$:  The edge $e=i j$ is oriented such that the smaller number between $i$ and $j$ is the tail. 
Let $N_0=N(D_0)$ be the directed incidence matrix of $D$.
 Consider an $ (n+m)\times 2 m$ integral matrix $E$ whose columns correspond
 the vectors $z(e),w(e')$  ($e\in E(G)$). With the right order of the columns, we see that 
$$ E=\begin{pmatrix}
N_0 & -N_0\\
I_m & I_m
\end{pmatrix} $$   
where $I_m$ is the $m\times m$ identity matrix. Now, the statement follows from Propositions \ref{construction}  and \ref{oriented Farkas graph}.

\end{proof}

In particular, we can apply Theorem \ref{Farkas, integers, first}, leading to the following theorem.
   
\begin{theorem}  \label{Landua 1}

Let $G$ be a connected graph  with $V(G)=\{1,...,n\}$ and $E(G)=\{e_1,...,e_m\}$. Let $r_i, s_j, a_{e}\leq b_{e}, c_{e}\leq d_{e}$ 
be integers where $i=1,...,n$, $j=1,...,m$ and $e\in E(G)$. Then there exist integers 
$a_{e}\leq x_{ e}\leq b_{e }$, $c_{e}\leq y_{ e}\leq d_{e }$ ($e\in E(G)$) such that 
$$\sum_{i=1}^n r_i f_i+\sum_{j=1}^m s_j g_j=\sum_{e\in G} x_{e} z(e)+\sum_{e\in G} y_{e} z'(e),$$ 
if and only if the following conditions hold:\\
(1) If $$\begin{pmatrix}
t_1\\
\vdots\\
t_n
\end{pmatrix}= N_0\begin{pmatrix}
s_1\\
\vdots\\
s_m
\end{pmatrix}+\begin{pmatrix}
r_1\\
\vdots\\
r_n
\end{pmatrix},$$  where $N_0$ is the matrix defined in Proposition \ref{directed Farkas graph 1}, then each $t_i$ is even and $\sum_{i=1}^n t_i=0$. \\
(2) For all sets $\emptyset \neq I\subsetneq  \{1,...,n\}$ and $J\subset E(I)$, such that the induced subgraphs of $G$ on $I$ and $\{1,...,n\}\setminus I$ are connected,
we have
$$\sum_{i\in I} r_i+\sum_{e_k\in E(I)\setminus J} s_k-\sum_{e_k\in J} s_k\leq  2\sum_{e=i j\notin J,i<j,i\in I,j\notin I} b_{e}- 2\sum_{e=i j\in J, i<j, i\notin I,j\in I} a_{e}$$
$$+ 2\sum_{e=i j\notin J,i<j,i\notin I,j\in I} d_{e}- 2\sum_{e=i j\in J, i<j, i\in I,j\notin I} c_{e}$$
(3)  For all edges $e=e_k\in E(G)$ such that the subgraph $G-e_k$ of $G$ is connected,
we have $ a_{e}+c_{e}\leq s_k\leq  b_{e}+d_{e}.$

\end{theorem}

\begin{proof}

We need to show that 
\begin{equation}\label{inside}
\sum_{i=1}^n r_i f_i+\sum_{j=1}^m s_j g_j\in \sum_{e\in E(G)}\mathbb{Z}z(e)+\sum_{e\in E(G)}\mathbb{Z}z'(e) 
\end{equation}
if and only if each $t_i$
 is even and $\sum_{i=1}^n t_i=0$. The rest of the proof is straightforward. 
Using the notations of Proposition \ref{directed Farkas graph 1}, we see that 
\ref{inside} holds
if and only if there exist vectors $X_1,X_2\in \mathbb{Z}^m$ such that 
$$ E\begin{pmatrix}
X_1\\
X_2
\end{pmatrix}= \begin{pmatrix}
r_1\\
\vdots\\
r_n\\
s_1\\
\vdots\\
s_m
\end{pmatrix}.$$   
This equation is equivalent to the following equations 
$$N_0X_1-N_0X_2=\begin{pmatrix}
r_1\\
\vdots\\
r_n
\end{pmatrix}, X_1+X_2=\begin{pmatrix}
s_1\\
\vdots\\
s_m
\end{pmatrix}.$$
These equations, in turn, are equivalent to the following equations 
$$2 N_0X_1=  N_0\begin{pmatrix}
s_1\\
\vdots\\
s_m
\end{pmatrix}+\begin{pmatrix}
r_1\\
\vdots\\
r_n
\end{pmatrix}=\begin{pmatrix}
t_1\\
\vdots\\
t_n
\end{pmatrix}, X_2=\begin{pmatrix}
s_1\\
\vdots\\
s_m
\end{pmatrix}-X_1.$$
But the last equations have a solution if and only if   each $t_i$ is even and $\sum_{i=1}^n t_i=0$, see the proof of Theorem \ref{Landua}.

\end{proof}

Let us now consider a special case of the above theorem. We recall that  
the score vector of a directed graph $D$ is by definition $(d_D^+(1),...,d_D^+(n))$, where we assume $V(D)=\{1,...,n\}$, as usual. 

\begin{corollary}\label{directed graphical sequences 1}

Let $G$ be a connected graph. A sequence $(r_1,...,r_n)$ of integers is a score vector of some orientation on $G$,  if and only if  
$\sum_{i=1}^n r_i=|E(G)|$ and we have $2\sum_{i\in I} r_i \leq |E(I)|+\sum_{i\in I} d_G(i)$
for all sets $\emptyset \neq I\subsetneq  \{1,...,n\}$, such that the induced subgraphs of $G$ on $I$ and $\{1,...,n\}\setminus I$ are connected.

\end{corollary}

\begin{proof}

For every orientation $D$ on $G$, from the identities $d_D^+(i)-d_D^-(i)=d_D^0(i), d_D^+(i)+d_D^-(i)=d_G(i)$,  we conclude that 
$d_D^0(i)=2 d_D^+(i)-d_G(i)$. So, a sequence $(r_1,...,r_n)$ of integers is a score vector of some orientation on $G$,  
if and only if the sequence $(2 r_1-d_G(1),...,2 r_n-d_G(n))$ is the signed degree sequence of some orientation on $G$. 
Now, using the notations in Theorem \ref{Landua 1}, one can easily show that  such an orientation on $G$ exists if and only if 
 there exist integers $0\leq x_{ e},y_e\leq 1$ ($e\in E(G)$) such that 
$$\sum_{i=1}^n (2 r_i-d_G(i)) f_i+\sum_{j=1}^m g_j=\sum_{e\in G} x_{e} z(e)+\sum_{e\in G} y_{e} z'(e).$$
Therefore, by setting $a_e=c_e=0$ and $b_e=d_e=1$ ($e\in E(G)$)  in Theorem  \ref{Landua 1}, we see that 
such an orientation on $G$ exists if and only if conditions (1), (2) and (3) in Theorem \ref{Landua 1} hold. Condition (3) trivially holds. Condition (2) is easily simplified and
we obtain  the desire inequalities. So it remains to see when the vector 
$$\begin{pmatrix}
t_1\\
\vdots\\
t_n
\end{pmatrix}= N_0\begin{pmatrix}
1\\
\vdots\\
1
\end{pmatrix}+\begin{pmatrix}
2 r_1-d_G(1)\\
\vdots\\
2 r_n- d_G(n)
\end{pmatrix},$$
satisfies the conditions in (1). But it is easy to see that $t_i=2 r_i -2d^-_{D_0}(i)$ (where $D_0$ is the orientation defined in Proposition \ref{directed Farkas graph 1})
and therefore the conditions in (1) are satisfied if and only if 
$\sum_{i=1}^n t_i=0 $ if and only if $\sum_{i=1}^n r_i=\sum_{i=1}^n d^-_{D_0}(i)=|E(G)|$.

\end{proof}

Note that if in the above corollary $G$ is the complete graph, then we obtain the well-known theorem of Landau regarding tournaments.

\end{subsection}

%%%%%%%%%%%%%%%%%%%%%%%%%%%%%%%%%%%%%%%%%%%%%%%%%%%%%%%%%%%%%%%%%%%%%%%%%%%%%%%%%%%%%%%%%%%%%%%%%%%%%%%%%%%
%%%%%%%%%%%%%%%%%%%%%%%%%%%%%%%%%%%%%%%%%%%%%%%%%%%%%%%%%%%%%%%%%%%%%%%%%%%%%%%%%%%%%%%%%%%%%%%%%%%%%%%%%%%

\end{section}

\end{document}